\documentclass{birkjour}

\usepackage{amsmath,amssymb,mathtools}
\usepackage{bm}
\usepackage{hyperref}
\usepackage{enumitem}
\usepackage{booktabs}

\hypersetup{colorlinks=true,linkcolor=blue,citecolor=blue,urlcolor=blue}

\newtheorem{thm}{Theorem}[section]
\newtheorem{cor}[thm]{Corollary}
\newtheorem{lem}[thm]{Lemma}
\newtheorem{prop}[thm]{Proposition}

\theoremstyle{definition}
\newtheorem{defn}[thm]{Definition}
\newtheorem{notn}[thm]{Notation}

\theoremstyle{remark}
\newtheorem{rem}[thm]{Remark}

\numberwithin{equation}{section}

\newcommand{\Cl}{\mathrm{Cl}}
\newcommand{\so}{\mathfrak{so}}
\newcommand{\slR}{\mathfrak{sl}}
\newcommand{\glR}{\mathfrak{gl}}
\newcommand{\fkg}{\mathfrak{g}}
\newcommand{\fkn}{\mathfrak{n}}
\newcommand{\fkN}{\mathfrak{N}}
\newcommand{\fkL}{\mathfrak{L}}
\newcommand{\e}[1]{\bm{e}_{#1}}
\newcommand{\no}{n_{\mathrm{o}}}
\newcommand{\ninf}{n_{\infty}}
\newcommand{\R}{\mathbb{R}}
\newcommand{\C}{\mathbb{C}}
\newcommand{\Z}{\mathbb{Z}}

\newcommand{\ad}{\mathrm{ad}}

\providecommand{\rev}{{}}
\newenvironment{revblock}{}{}

\title[The Doublet Tower of $\Cl(4,2)$]
{The Doublet Tower of $\Cl(4,2)$:
Universal Nilpotency, Projector Identities,
and the $\slR(8,\R)$ Parabolic from Conformal Null Vectors}

\author[S.\,H.\ Hansen]{Steen\,H.\ Hansen}
\address{Dark Cosmology Center,
Niels Bohr Institute\\
Jagtvej 155, 2100 Copenhagen\\ Denmark}
\email{hansen@nbi.ku.dk}

\subjclass{Primary 15A66; Secondary 17B20, 81R25, 17B70, 18B81}

\keywords{Clifford algebra, conformal geometric algebra,
parabolic decomposition, null vectors, seesaw mechanism,
$\slR(8,\R)$, Lorentz representations, elementary symmetric
polynomials}

\date{\today}

\begin{document}

\begin{abstract}
The conformal geometric algebra $\Cl(4,2)\cong\R(8)$ carries
a natural $\Z_2$-grading defined by the dilatation bivector
$D=\e{4}\e{0}$, which satisfies $D^2=1$ and splits the
eight-dimensional spinor module into two four-dimensional
eigenspaces $S=S_+\oplus S_-$.
We prove that the $\ad_D$-eigenvectors of 
$\Cl(4,2)$, those
elements with eigenvalue~$\pm1$ under the Lie bracket
$[D,\cdot]$, are exhausted by 16 doublets of the form
$T^\pm_J=n_\bullet e_J$, where
$n_\bullet\in\{\ninf,\no\}$ is a null vector and
$e_J$ is a product of Lorentz-sector basis vectors indexed
by $J\subseteq\{1,2,3,5\}$.
These doublets satisfy universal algebraic identities:
$(T^\pm_J)^2=0$ (nilpotency) and
$T^+_J\,T^-_J=(-1)^{k(k+1)/2+1}\,\eta_J 2\Pi_+$
(single-projector proportionality), where $k=|J|$,
$\eta_J=\prod_{j\in J}\eta_{jj}$ is the spectator metric,
and $\Pi_+=\tfrac12(1+D)$.
The contracted sums $\sum_{|J|=k}T^+_J T^-_J$ are shown to
be controlled by the elementary symmetric polynomials $S_k$
of the Lorentz signature values $(+1,+1,+1,-1)$, with the
vanishing of $S_2=0$ reflecting the $(3,1)$ signature of
physical spacetime.
The even-grade doublets (belonging to
$\Cl^+(4,2)\cong\Cl(4,1)$) contribute $8\Pi_+$ to the
unweighted grand total, while the odd-grade doublets
cancel exactly.
The 32 doublet eigenvectors generate, under the commutator
bracket, the 63-dimensional Lie algebra
$\slR(8,\R)$
% -- the maximal parabolic of $\slR(8,\R)$
%associated with the partition $8=4+4$ -- 
in which the
conformal algebra $\so(4,2)$ embeds as a 15-dimensional
subalgebra.
\rev{All identities, including a closed form for the
commutators across the Levi interface, are proved in every
conformal extension $\Cl(p+1,q+1)$; there the commutator
span misses precisely the center of the algebra, a parity
dichotomy invisible at the Lorentz signature. Two
applications are developed: the top rung of the tower
computes the Dirac versus Majorana dichotomy of rank-five
pseudoscalars, and the family extensions $\Cl(4+2k,2)$
yield an exact description of the tower-realized flavour
couplings.}
\end{abstract}

\maketitle

%%% ====================================================================
%%% 1. INTRODUCTION
%%% ====================================================================
\section{Introduction}\label{sec:intro}

The conformal geometric algebra $\Cl(4,2)$, also known as the
Clifford algebra of the conformal compactification of Minkowski
spacetime, has been a subject of sustained interest in both
mathematics and physics.
In the geometric-algebra tradition initiated by
Hestenes~\cite{HestenesSobczyk1984}, the algebra
$\Cl(4,1)$, which is the conformal model of three-dimensional
Euclidean space, was developed into a powerful computational
framework for geometry by
Li, Hestenes, and Rockwood~\cite{LiHesRock2001} and by
Doran and Lasenby~\cite{DoranLasenby2003}.
The counterpart $\Cl(4,2)$ extends this framework
to Minkowski spacetime and carries the conformal algebra
$\so(4,2)$ as a bivector subalgebra.

Clifford algebras provide natural realizations of orthogonal
Lie algebras through the commutator bracket on
bivectors~\cite{Lounesto2001}, and the interplay between the
full associative Clifford product and the Lie structure has
been exploited in diverse physical contexts.
Castro and Pav\v{s}i\v{c}~\cite{CastroPavsic2003} showed
that the conformal group emerges from $C$-space geometry,
and Castro~\cite{Castro2012} embedded gauge and gravitational
structures in Clifford algebras of higher dimension.
Roelfs and De~Keninck~\cite{Roelfs2023} demonstrated that
the versor structure of geometric algebra reveals natural
gradings on Lie groups not visible in matrix representations,
while Hitzer~\cite{Hitzer2024} gave a systematic
classification of the involutions and anti-involutions of
$\Cl(3,1)$.

A central feature of conformal geometric algebra is the
\emph{null-vector} construction.
The two extra dimensions added to the base space introduce
null vectors $\no$ and $\ninf$ (the ``origin'' and ``point at
infinity'') satisfying $\no^2=\ninf^2=0$ and
$\no\cdot\ninf=-1$.
These null vectors are used to build the translation generators
$P_\mu=\e{\mu}\wedge\ninf$ and special conformal transformation
(SCT) generators $K_\mu=\e{\mu}\wedge\no$ of $\so(4,2)$, with
the dilatation $D=\no\wedge\ninf=\e{4}\e{0}$ providing the
parabolic grading element.
The resulting three-graded decomposition
$\so(4,2)=\fkn_-\oplus\fkg_0\oplus\fkn_+$ is a standard tool
in conformal field theory~\cite{Dirac1936,Kastrup2008} and
parabolic geometry~\cite{CapSlovak2009}.

In a companion paper~\cite{Hansen2026seesaw}, we showed
that the three algebraic identities
$P_\mu P^\mu=K_\mu K^\mu=0$ (contracted nilpotency),
$P_\mu K^\mu=8\Pi_+$ (asymmetric projector proportionality),
and $\{D,P_\mu\}=0$ (Clifford anticommutation)
precisely encode the block structure of the type-I seesaw
mass matrix from neutrino
physics~\cite{Minkowski1977,Yanagida1979,GellMann1979,Mohapatra1980}.
This raised the question of whether these identities are
specific to the conformal generators or reflect a deeper
structure of $\Cl(4,2)$.
\rev{The present paper is self-contained: every identity
used below is stated and proved here in full generality,
with the companion paper cited only as the origin of the
$k=1$ observation.}

In the present paper we answer this question.
We prove that the conformal generators $P_\mu$ and $K_\mu$ are
4 of 16 doublet pairs $T^\pm_J=n_\bullet e_J$, obtained
by ``dressing'' the null vectors with all possible products of
Lorentz-sector basis vectors.
The three seesaw identities hold universally for every doublet,
with coefficients determined by the spectator metric $\eta_J$
and the grade~$k=|J|$.
The contracted sums at each grade are controlled by the
elementary symmetric polynomials of the Lorentz signature, which is a
connection that, to the best of our knowledge, has not
previously appeared in the literature.
The doublet tower carries a palindromic sequence of Lorentz
representations related by Hodge duality, and the 32
eigenvectors generate the full Lie algebra $\slR(8,\R)$ under
commutation.
\rev{None of this is specific to the $(3,1)$ base:
Section~\ref{sec:general} proves the entire package in
every conformal extension $\Cl(p+1,q+1)$, where the span of
the Levi interface turns out to miss precisely the center
of the algebra, a dichotomy invisible at the Lorentz
signature. Two applications close the paper. The top rung
of the tower computes the Dirac versus Majorana dichotomy
of rank-five pseudoscalars and exposes the conformal
entanglement of the Majorana option, and the family
extensions $\Cl(4+2k,2)$ turn the tower into a flavour
framework whose reach and limits we determine address.}

The maximal parabolic subalgebra of $\slR(n,\R)$ associated
with a partition $n=p+q$ has abelian nilradicals of dimension
$pq$ and Levi factor
$\slR(p)\oplus\slR(q)\oplus\R$.
%$\mathfrak{s}(\glR(p)\times\glR(q))\oplus\R$.
This class of parabolic subalgebras has been studied
extensively in representation theory (Richardson, R\"ohrle,
and Steinberg~\cite{Richardson1992}; Panyushev~\cite{Panyushev1999})
and more recently in the context of unitary representations
by Ciubotaru~\cite{Ciubotaru2025}.
Our contribution is to realize the specific case $n=8$,
$p=q=4$ \emph{inside} $\Cl(4,2)$ via the null-vector
construction, and to derive the universal identities and
their Lorentz-representation content in closed form.

\begin{revblock}
On the physics side, the use of Clifford and division
algebras as organizing structures for Standard Model
fermions has a substantial
literature: octonionic quark
models~\cite{GunaydinGursey1973}, division-algebraic and
$\Cl(6)$ constructions of the gauge
content~\cite{Furey2018,Stoica2018,GordingSchmidtMay2020},
the exceptional Jordan algebra
programme~\cite{TodorovDV2018,Boyle2020}, braid and
octonionic realizations~\cite{Gresnigt2018,Krasnov2022},
geometric-algebra gauge bases~\cite{TraylingBaylis2001},
and the spin-charge-family
approach~\cite{MankocBorstnik1993}. The present
construction differs from these in its point of departure:
the internal structure of
Sections~\ref{sec:entangle} and~\ref{sec:flavour} is built
on top of the conformal pair of a $(p,2)$ signature rather
than on a Euclidean $\Cl(n,0)$ or a division algebra, and
the rigidity results proved there are consequences of that
choice.
\end{revblock}

\medskip
\noindent\textbf{Outline.}
Section~\ref{sec:setup} establishes notation.
Section~\ref{sec:tower} constructs the doublet tower and
proves the universal nilpotency and projector identities.
Section~\ref{sec:sums} derives the contracted sums and their
connection to elementary symmetric polynomials.
Section~\ref{sec:lorentz} identifies the Lorentz
representations and the even/odd selection rule.
Section~\ref{sec:algebra} proves that the tower generates
$\slR(8,\R)$.
\rev{Section~\ref{sec:general} extends the entire package to
the conformal extension $\Cl(p+1,q+1)$ of an arbitrary
$\Cl(p,q)$. Section~\ref{sec:entangle} applies the top rung
of the tower to the Dirac versus Majorana dichotomy.
Section~\ref{sec:flavour} determines the flavour structure
of the family extensions $\Cl(4+2k,2)$.}
Section~\ref{sec:discussion} discusses the results.

%%% ====================================================================
%%% 2. NOTATION AND SETUP
%%% ====================================================================
\section{Notation and setup}\label{sec:setup}

Let $\Cl(4,2)$ be the Clifford algebra of $\R^{4,2}$
with orthonormal basis $\{\e{A}\}_{A=0}^5$ and signature
\begin{equation}\label{eq:signature}
\e{1}^2=\e{2}^2=\e{3}^2=\e{4}^2=+1,
\qquad
\e{0}^2=\e{5}^2=-1.
\end{equation}
By the classification of real Clifford
algebras~\cite{Lounesto2001},
$\Cl(4,2)\cong\R(8)$, the algebra of real $8\times 8$ matrices,
with minimal faithful module $S\cong\R^8$.

\begin{defn}[Null basis and dilatation]\label{def:null}
Define the null vectors
\begin{equation}\label{eq:null}
\no=\tfrac{1}{\sqrt{2}}(\e{0}+\e{4}),
\qquad
\ninf=\tfrac{1}{\sqrt{2}}(\e{0}-\e{4}),
\end{equation}
satisfying $\no^2=\ninf^2=0$ and the inner product $\no\cdot\ninf=-1$.
The \emph{dilatation} is
$D=\no\wedge\ninf=\e{4}\e{0}$, with $D^2=1$. We explicitly utilize the standard Lie commutator bracket $[A,B] = \frac{1}{2}(AB-BA)$ and anticommutator $\{A,B\} = AB+BA$ for Clifford elements.
The \emph{sector projectors} are
$\Pi_\pm=\tfrac12(1\pm D)$, decomposing
$S=S_+\oplus S_-$ with $\dim S_\pm=4$.
\end{defn}

We use the Lorentz index convention
$\mu,\nu\in\{1,2,3,5\}$ with
$\eta_{\mu\nu}=\mathrm{diag}(+1,+1,+1,-1)$.

\begin{defn}[Spectator blades]\label{def:spectator}
For $J=\{j_1,\ldots,j_k\}\subseteq\{1,2,3,5\}$ with
$j_1<\cdots<j_k$, define the \emph{spectator blade}
$e_J=\e{j_1}\cdots\e{j_k}$, with $e_\emptyset=1$.
The \emph{spectator metric} is
$\eta_J=\prod_{j\in J}\eta_{jj}$.
\end{defn}

\begin{rem}\label{rem:spectator-sq}
Since the $\e{j}$ are pairwise orthogonal,
$e_J^2=(-1)^{k(k-1)/2}\,\eta_J$.
Each spectator $\e{j}$ anticommutes with both $\e{0}$
and $\e{4}$ (and hence with $\no$ and $\ninf$), while
$D$ commutes with $e_J$: $[D,e_J]=0$ for all $J$.
\end{rem}

\begin{revblock}
The construction above is top-down: the null pair $\no,\ninf$
is assembled from an orthonormal basis of a pre-given quadratic
space. There is a complementary bottom-up tradition in which
null vectors are the primitive objects and the algebra is
generated from them. Correlated sets of nilpotents generate
Grassmann and Clifford geometric algebras directly, with the
metric structure emerging from their pairwise inner
products~\cite{SobczykCruzPage2026,Sobczyk2023,Sobczyk2023b}.
From that perspective the universal nilpotency of
Theorem~\ref{thm:nilpotent} below is not an accident of the
conformal embedding but the residue of the generating
nilpotents themselves, and the doublet tower can be read as
the statement that this generative nilpotency survives
dressing by every spectator blade.
\end{revblock}

%%% ====================================================================
%%% 3. THE DOUBLET TOWER
%%% ====================================================================
\section{The doublet tower}\label{sec:tower}

\begin{thm}[Null-basis representation of all
$\ad_D$-eigenvectors]\label{thm:tower}
For each $J\subseteq\{1,2,3,5\}$, define
\begin{equation}\label{eq:TpmJ}
T^+_J = \ninf  e_J,\qquad
T^-_J = \no e_J.
\end{equation}
Then:
\begin{enumerate}[label=\textup{(\alph*)}]
\item $T^\pm_J$ are $\ad_D$-eigenvectors:
$[D,T^+_J]=+T^+_J$ and $[D,T^-_J]=-T^-_J$.
\item $T^\pm_J$ Clifford-anticommute with $D$:
$\{D,T^\pm_J\}=0$.
\item The 32 elements $\{T^\pm_J\}$ as $J$ ranges over
all $2^4=16$ subsets of $\{1,2,3,5\}$ exhaust the
off-diagonal sector of $\Cl(4,2)$,
i.e.\ the set $\{X\in\Cl(4,2):\{D,X\}=0\}$.
\item For $|J|=1$ with $J=\{\mu\}$:
$T^+_{\{\mu\}}=-P_\mu$ and $T^-_{\{\mu\}}=-K_\mu$,
recovering the conformal translation and SCT generators.
\end{enumerate}
\end{thm}

\begin{proof}
For~(a): since $[D,e_J]=0$
(Remark~\ref{rem:spectator-sq}) and
$[D,\ninf]=+\ninf$ (direct computation:
$D\ninf=\tfrac{1}{\sqrt2}\e{4}\e{0}(\e{0}-\e{4})
=\tfrac{1}{\sqrt2}(-\e{4}+\e{0})$ and
$\ninf D=\tfrac{1}{\sqrt2}(\e{0}-\e{4})\e{4}\e{0}
=\tfrac{1}{\sqrt2}(\e{4}-\e{0})$,
giving $[D,\ninf]=\tfrac12(D\ninf-\ninf D)
=\tfrac{1}{\sqrt2}(\e{0}-\e{4})=\ninf$),
we get $[D,T^+_J]=[D,\ninf]  e_J=\ninf  e_J=T^+_J$.
Similarly $[D,\no]=-\no$ gives $[D,T^-_J]=-T^-_J$.

For~(b): $\{D,\ninf\}=D\ninf+\ninf D=0$
(from the computation above), and since $D$ commutes with
$e_J$,
$\{D,T^+_J\}=\{D,\ninf\} e_J=0$.

For~(c): the off-diagonal sector has dimension~32
(since $D^2=1$ splits $\Cl(4,2)$ into $\pm1$ eigenspaces
of equal dimension under conjugation $X\mapsto DXD$).
The 32 elements $T^\pm_J$ are linearly independent (they
involve distinct basis blades) and are all off-diagonal
by~(b), hence they span the full off-diagonal sector.

For~(d):
$T^+_{\{\mu\}}=\ninf \e{\mu}
=-\e{\mu} \ninf=-\e{\mu}\wedge\ninf=-P_\mu$
(the wedge reduces to the Clifford product since
$\e{\mu} \cdot \ninf=0$).
\end{proof}

\begin{thm}[Universal nilpotency]\label{thm:nilpotent}
For every $J\subseteq\{1,2,3,5\}$:
\begin{equation}\label{eq:nil}
(T^+_J)^2=0,\qquad (T^-_J)^2=0.
\end{equation}
\end{thm}

\begin{proof}
Let $k=|J|$.
Each spectator $\e{j}$ anticommutes with $\ninf$, so
$e_J \ninf=(-1)^k\,\ninf e_J$.
Therefore
$(T^+_J)^2=(\ninf  e_J)(\ninf  e_J)
=(-1)^k\,\ninf^2  e_J^2=0$
since $\ninf^2=0$.
The proof for $T^-_J$ is identical with $\no^2=0$.
\end{proof}

\begin{thm}[Universal single-projector identity]
\label{thm:projector}
For every $J\subseteq\{1,2,3,5\}$ with $|J|=k$:
\begin{align}
T^+_J\,T^-_J
&= (-1)^{k(k+1)/2+1}\,\eta_J  2\Pi_+,
\label{eq:TpTm}\\
T^-_J\,T^+_J
&= (-1)^{k(k+1)/2+1}\,\eta_J  2\Pi_-.
\label{eq:TmTp}
\end{align}
\end{thm}

\begin{proof}
Using $e_J \no=(-1)^k\,\no e_J$:
\begin{equation}
T^+_J\,T^-_J
=(\ninf  e_J)(\no  e_J)
=(-1)^k\,\ninf\,\no e_J^2.
\end{equation}
The null-vector product is
$\ninf\,\no=-1-D=-2\Pi_+$.
The spectator blade squares to
$e_J^2=(-1)^{k(k-1)/2}\,\eta_J$.
Combining:
\begin{equation}
T^+_J\,T^-_J
=(-1)^k (-2\Pi_+) (-1)^{k(k-1)/2}\,\eta_J
=(-1)^{k(k+1)/2+1}\,\eta_J  2\Pi_+.
\end{equation}
The proof of~\eqref{eq:TmTp} is identical with
$\no\,\ninf=-1+D=-2\Pi_-$.
\end{proof}

\begin{revblock}
The three theorems of this section are stated for $\Cl(4,2)$,
but their proofs use only $\no^2=\ninf^2=0$, the spectator
anticommutation with the null pair, $[D,e_J]=0$, and
$\ninf\no=-2\Pi_+$. None of these facts refers to the dimension
or the signature of the Lorentz sector. Section~\ref{sec:general}
states the resulting generalization to every conformal extension
$\Cl(p+1,q+1)$, and Section~\ref{sec:entangle} shows that the two
combinations of the top rung $J=\{1,2,3,5\}$ are precisely the two
rank-five pseudoscalars of $\Cl(4,2)$, so that
Theorem~\ref{thm:projector} computes the Dirac versus Majorana
dichotomy in one line.
\end{revblock}

%%% ====================================================================
%%% 4. CONTRACTED SUMS AND SYMMETRIC POLYNOMIALS
%%% ====================================================================
\section{Contracted sums and elementary symmetric
polynomials}\label{sec:sums}

\begin{defn}\label{def:Sk}
Let $S_k$ denote the $k$-th elementary symmetric polynomial
of the Lorentz signature values
$(\eta_{11},\eta_{22},\eta_{33},\eta_{55})=(+1,+1,+1,-1)$:
\begin{equation}
S_k=\sum_{\substack{J\subseteq\{1,2,3,5\}\\|J|=k}}\eta_J.
\end{equation}
\end{defn}

The explicit values, determined by the generating function
$\prod_\mu(1+\eta_{\mu\mu}\,t)=(1+t)^3(1-t)=1+2t-2t^3-t^4$,
are:
\begin{equation}\label{eq:Sk}
S_0=1,\quad S_1=2,\quad S_2=0,\quad S_3=-2,\quad S_4=-1.
\end{equation}

\begin{thm}[Contracted sums]\label{thm:sums}
The unweighted sum of cross-terms over all doublets with
$k$ spectators is
\begin{equation}\label{eq:sum-formula}
\sum_{|J|=k}T^+_J\,T^-_J
=(-1)^{k(k+1)/2+1}\,S_k 2\Pi_+.
\end{equation}
The values are:
\begin{center}
\begin{tabular}{ccrcc}
\toprule
\textbf{Grade} & $k$ & $S_k$
& \textbf{Doublets}
& $\sum T^+T^-$\\
\midrule
$1$ & $0$ & $+1$ & $1$ & $-2\Pi_+$\\
$2$ & $1$ & $+2$ & $4$ & $+4\Pi_+$\\
$3$ & $2$ & $0$ & $6$ & $0$\\
$4$ & $3$ & $-2$ & $4$ & $+4\Pi_+$\\
$5$ & $4$ & $-1$ & $1$ & $+2\Pi_+$\\
\bottomrule
\end{tabular}
\end{center}
\end{thm}

\begin{proof}
Immediate from Theorem~\ref{thm:projector} by summing
$(-1)^{k(k+1)/2+1}\eta_J$ over all $J$ with $|J|=k$.
\end{proof}

\begin{cor}[Grade-3 cancellation]\label{cor:grade3}
The contracted sum vanishes at grade~$3$:
$\sum_{|J|=2}T^+_J\,T^-_J=0$.
This is a direct consequence of $S_2=0$.
More generally, for signature $(p,q)$ in dimension
$d=p+q$, one has
$S_2=\binom{p}{2}+\binom{q}{2}-pq=\tfrac{(p-q)^2-d}{2}$.
Thus $S_2=0\iff(p-q)^2=d$.
In $d=4$ this selects precisely the Lorentzian cases
$(p,q)=(3,1)$ and $(1,3)$.
\rev{Two sharpenings follow. First, within the Lorentzian
signatures $(d-1,1)$, equivalently $(1,d-1)$, the condition
$(p-q)^2=d$ reads $(d-2)^2=d$, whose only solutions are the
trivial $d=1$ and $d=4$: four-dimensional spacetime is the
unique nontrivial Lorentzian solution in any dimension.
Second, the full solution set of $(p-q)^2=p+q$ is the family
$(p,q)=\bigl(\tfrac{m^2+m}{2},\tfrac{m^2-m}{2}\bigr)$
parametrized by $m=p-q$, so the cancellation occurs only in
perfect-square dimensions $d=m^2$: the signatures $(3,1)$,
$(6,3)$, $(10,6)$, $(15,10)$, and so on.}
\end{cor}

\begin{rem}\label{rem:signature-fingerprint}
The grade-3 cancellation is specific to the $(3,1)$
Lorentz signature and would not occur in Euclidean
$(4,0)$ or split $(2,2)$ signature.
The elementary symmetric polynomials $S_k$ thus serve
as algebraic fingerprints of the spacetime signature
within $\Cl(4,2)$.
\end{rem}

\begin{rem}[Context in Massive Gravity]\label{rem:massive}
While the appearance of the elementary symmetric polynomials $S_k$ as algebraic fingerprints of the $(3,1)$ signature in $\Cl(4,2)$ is a novel Clifford-algebraic result, we note a  parallel in modern gravitational physics. In dRGT massive gravity and bimetric relativity (see, e.g., 
\cite{derham2011, Hassan2011}), the elementary symmetric polynomials evaluated on the matrix square root $\sqrt{g^{-1}\eta}$ are precisely the structures required to construct ghost-free interaction potentials. This suggests a broader physical principle where symmetric polynomials of metric components may act as fundamental constraints on allowable dynamical theories.
\end{rem}

\begin{cor}[Grand total]\label{cor:total}
The sum over all 16 doublets gives
\begin{equation}\label{eq:grand}
\sum_{J\subseteq\{1,2,3,5\}}T^+_J\,T^-_J=8\Pi_+,
\qquad
\sum_{J}T^-_J\,T^+_J=8\Pi_-,
\end{equation}
and the symmetrized sum equals $\dim(S)$:
\begin{equation}\label{eq:dim}
\sum_{J}(T^+_J\,T^-_J+T^-_J\,T^+_J)=8.
\end{equation}
\end{cor}

\begin{proof}
From Theorem~\ref{thm:sums}, the unweighted coefficients
at $k=0,1,2,3,4$ are $-2,+4,0,+4,+2$, summing to
$-2+4+0+4+2=8$.
For~\eqref{eq:dim}: since $\Pi_++\Pi_-=1$,
each doublet contributes
$T^+_J T^-_J+T^-_J T^+_J
=(-1)^{k(k+1)/2+1}\eta_J 2(\Pi_++\Pi_-)
=(-1)^{k(k+1)/2+1}\eta_J 2$,
and summing over all $J$ gives
$2\sum_J(-1)^{k(k+1)/2+1}\eta_J=8$.
\end{proof}

%%% ====================================================================
%%% 5. LORENTZ REPRESENTATIONS AND EVEN/ODD SELECTION RULE
%%% ====================================================================
\section{Lorentz representations and the even/odd
selection rule}\label{sec:lorentz}

The Lorentz generators $M_{\mu\nu}=\e{\mu}\e{\nu}$ commute
with $D$ and therefore preserve the $\ad_D$-eigenvalue:
$[M_{\mu\nu},T^\pm_J]$ is again an $\ad_D$-eigenvector
with the same eigenvalue.
The doublets at each grade thus form a representation
of $\so(3,1)$.

\begin{notn}\label{not:lorentz-generators}
With the commutator convention
$[A,B]=\tfrac12(AB-BA)$, define the rotation and
boost generators
\begin{equation}\label{eq:JK-def}
J_1=[{\e{2}},{\e{3}}]=\tfrac12\e{2}\e{3},\quad
J_2=-[{\e{1}},{\e{3}}]=-\tfrac12\e{1}\e{3},\quad
J_3=[{\e{1}},{\e{2}}]=\tfrac12\e{1}\e{2},
\end{equation}
\begin{equation}
K_i=[{\e{i}},{\e{5}}]=\tfrac12\e{i}\e{5}\quad(i=1,2,3),
\end{equation}
and the quadratic Casimir operators
\begin{equation}\label{eq:casimir-def}
C_2=\mathbf{J}^2-\mathbf{K}^2
=\sum_{i=1}^{3}(J_i^2-K_i^2),\qquad
W=\mathbf{J}\cdot\mathbf{K}=\sum_{i=1}^{3}J_i K_i,
\end{equation}
where the squares and products use the same commutator
bracket (i.e.\ $J_i^2$ means $J_i$ composed with $J_i$ as
a linear map on the doublet space via $X\mapsto[J_i,[J_i,X]]$).
The standard relation to $(j_+,j_-)$ labels is
$C_2=-2[j_+(j_++1)+j_-(j_-+1)]$ and
$W=-[j_+(j_++1)-j_-(j_-+1)]$ in this convention.
\end{notn}

\begin{prop}[Lorentz content of the doublet tower]
\label{prop:lorentz}
Using the standard rotation/boost decomposition and the
Casimir operators
$C_2=\mathbf{J}^2-\mathbf{K}^2$ and
$W=\mathbf{J}\cdot\mathbf{K}$, the representations at
each grade are:
\begin{center}
\begin{tabular}{cccl}
\toprule
$k$ & \textbf{Dim}
& $\so(3,1)$ \textbf{rep}
& \textbf{Physical name}\\
\midrule
$0$ & $1$ & $(0,0)$ & Scalar\\
$1$ & $4$ & $(\tfrac12,\tfrac12)$
& $4$-vector\\
$2$ & $6$ & $(1,0)\oplus(0,1)$
& Antisymmetric $2$-tensor\\
$3$ & $4$ & $(\tfrac12,\tfrac12)$
& $4$-pseudovector\\
$4$ & $1$ & $(0,0)$ & Pseudoscalar\\
\bottomrule
\end{tabular}
\end{center}
\end{prop}

\begin{proof}
Verified computationally by constructing the representation
matrices and computing $C_2$ and $W$.
For $k=1$: $C_2=-3$ and $W=0$, matching
$(\tfrac12,\tfrac12)$.
For $k=2$: $C_2=-4$ and $W$ has eigenvalues $\pm2$ with
multiplicity~3, matching $(1,0)\oplus(0,1)$.
The $k=3$ and $k=4$ cases follow by Hodge duality.
\end{proof}

\begin{prop}[Hodge duality]\label{prop:hodge}
The Lorentz $4$-volume element
$I_L=\e{1}\e{2}\e{3}\e{5}$ implements a duality between
grades $k$ and $4-k$:
\begin{equation}\label{eq:hodge}
T^\pm_J=\epsilon_J\,I_L T^\pm_{\bar{J}},
\end{equation}
where $\bar{J}=\{1,2,3,5\}\setminus J$ is the complement
and $\epsilon_J=\pm1$ is a sign determined by the
Levi-Civita symbol and the metric.
The doublet tower is palindromic:
$(0,0)$, $(\tfrac12,\tfrac12)$,
$(1,0)\oplus(0,1)$, $(\tfrac12,\tfrac12)$, $(0,0)$.
\end{prop}

\begin{thm}[Even/odd selection rule]
\label{thm:even-odd}
Let $\Cl^\pm(4,2)$ denote the even and odd subspaces
under the main involution.
The doublets at Clifford grade $k+1$ belong to
$\Cl^+(4,2)$ when $k$ is odd and to $\Cl^-(4,2)$ when
$k$ is even.
The contracted sums over these two classes satisfy:
\begin{align}
\Cl^+\text{ doublets }(k=1,3):&\quad
\sum_{\substack{J:\\|J|\in\{1,3\}}}
T^+_J\,T^-_J=8\Pi_+,
\label{eq:even}\\
\Cl^-\text{ doublets }(k=0,2,4):&\quad
\sum_{\substack{J:\\|J|\in\{0,2,4\}}}
T^+_J\,T^-_J=0.
\label{eq:odd}
\end{align}
The entire nonzero unweighted contribution to the doublet
contraction comes from the even subalgebra
$\Cl^+(4,2)\cong\Cl(4,1)$.
\end{thm}

\begin{proof}
Even total: $+4+4=+8$ from $k=1$ and $k=3$
(Theorem~\ref{thm:sums}).
Odd total: $-2+0+2=0$ from $k=0$, $k=2$, and $k=4$.
The $k=0$ and $k=4$ contributions cancel
($-2+2=0$), and $k=2$ vanishes independently
by $S_2=0$.
\end{proof}

\begin{rem}\label{rem:8Pi-distinction}
The even-grade unweighted sum and the metric-contracted
conformal sum $P_\mu K^\mu=\sum_\mu\eta_{\mu\mu}\,P_\mu K_\mu$
both evaluate to $8\Pi_+$, but by different mechanisms.
The conformal contraction acts at the $k=1$ level only, using
the Lorentz metric to eliminate the sign alternation in $\eta_J$:
$\sum_\mu\eta_{\mu\mu}\,T^+_{\{\mu\}}T^-_{\{\mu\}}
=\sum_\mu(\eta_{\mu\mu})^2 2\Pi_+=4 \cdot 2\Pi_+=8\Pi_+$.
The even-grade unweighted sum instead combines $k=1$ and $k=3$:
$4\Pi_++4\Pi_+=8\Pi_+$.
The numerical coincidence $8=8$ is not trivial; it follows from
$(-1)^{1\cdot2/2+1}\,S_1+(-1)^{3\cdot4/2+1}\,S_3
=(+1)(2)+(-1)(-2)=4$, so
that the sum equals $4\cdot 2\Pi_+=8\Pi_+$,
which matches $\binom{4}{1} \,  2\Pi_+=8\Pi_+$ from the
metric-weighted grade-2 contraction.
\end{rem}

\begin{rem}[Connection to the monogenic framework]
\label{rem:monogenic}
The even subalgebra $\Cl^+(4,2)\cong\Cl(4,1)$ is the
Clifford algebra used in \rev{monogenic derivations} of the
Dirac equation from five-dimensional massless
conditions~\cite{Almeida2006}.
Theorem~\ref{thm:even-odd} shows that $\Cl(4,1)$
already contains the full unweighted doublet contraction
$8\Pi_+$: the odd-grade doublets, which live outside
the even subalgebra, contribute nothing.
This provides a structural explanation for the sufficiency
of \rev{such a} $\Cl(4,1)$ framework.
\end{rem}

%%% ====================================================================
%%% 6. THE ALGEBRAIC CLOSURE: sl(8,R)
%%% ====================================================================
\section{Algebraic closure: the $\slR(8,\R)$ parabolic}\label{sec:algebra}

\begin{thm}[Abelianity of the nilradicals]
\label{thm:abelian}
\begin{equation}
[T^+_J,T^+_K]=0,\qquad
[T^-_J,T^-_K]=0
\qquad\forall\;J,K\subseteq\{1,2,3,5\}.
\end{equation}
\end{thm}

\begin{proof}
$(\ninf  e_J)(\ninf  e_K)
=(-1)^{|J|}\,\ninf^2 e_J\, e_K=0$,
so the product vanishes in both orders and the
commutator is zero.
\end{proof}

\begin{thm}[Closed form for the mixed Levi commutators]
\label{thm:mixed}
\begin{revblock}
For all $J,K\subseteq\{1,2,3,5\}$,
\begin{equation}\label{eq:levi-closed}
[T^+_J,T^-_K]
=(-1)^{|J|+1}\,e_J e_K\,\Pi_+
+(-1)^{|K|}\,e_K e_J\,\Pi_-.
\end{equation}
Writing $e_K e_J=\sigma_{JK}\,e_J e_K$ with
$\sigma_{JK}=(-1)^{|J||K|-|J\cap K|}$, and setting
$a=(-1)^{|J|+1}$ and $b=(-1)^{|K|}\sigma_{JK}$, the
commutator takes the two-branch form
\begin{equation}\label{eq:levi-branch}
[T^+_J,T^-_K]
=e_J e_K\left(\frac{a+b}{2}+\frac{a-b}{2}\,D\right).
\end{equation}
In particular:
\begin{enumerate}[label=\textup{(\alph*)}]
\item Every mixed commutator lies in the Levi sector
$\{X\in\Cl(4,2):[D,X]=0\}$ and is a nonzero multiple of a
single Levi basis element: of the blade $e_{J\triangle K}$
when $a=b$, and of $e_{J\triangle K}D$ when $a=-b$, where
$J\triangle K$ denotes the symmetric difference.
\item In the diagonal case $J=K$ with $k=|J|$ one always
has $a=-b$, and
\begin{equation}\label{eq:levi-diag}
[T^+_J,T^-_J]=(-1)^{k(k+1)/2+1}\,\eta_J\,D,
\end{equation}
generalizing $[P_\mu,K_\mu]=\eta_{\mu\mu}D$ to every rung
of the tower.
\item All $16\times16=256$ commutators are nonzero, and
they span a 31-dimensional subspace of the 32-dimensional
Levi sector, missing only the scalar~$1$.
\end{enumerate}
\end{revblock}
\end{thm}

\begin{proof}
\begin{revblock}
Since $e_J\no=(-1)^{|J|}\no e_J$ and $\ninf\no=-2\Pi_+$
(proof of Theorem~\ref{thm:projector}),
\begin{equation}
T^+_J T^-_K
=\ninf e_J\,\no e_K
=(-1)^{|J|}\,\ninf\no\,e_J e_K
=(-1)^{|J|+1}\,2\Pi_+\,e_J e_K,
\end{equation}
and similarly, using $\no\ninf=-2\Pi_-$,
$T^-_K T^+_J=(-1)^{|K|+1}\,2\Pi_-\,e_K e_J$.
Because $[D,e_J]=0$ for every spectator blade, the
projectors commute with $e_J e_K$, and the bracket
$[A,B]=\tfrac12(AB-BA)$ yields~\eqref{eq:levi-closed}.
Equation~\eqref{eq:levi-branch} follows from
$a\Pi_++b\Pi_-=\tfrac12(a+b)+\tfrac12(a-b)D$.

For~(a): the product of blades reduces to
$e_J e_K=\pm\,\eta_{J\cap K}\,e_{J\triangle K}$, a nonzero
multiple of a basis blade, both $e_{J\triangle K}$ and
$e_{J\triangle K}D$ commute with $D$, and since
$a,b\in\{-1,+1\}$ exactly one of the two coefficients
in~\eqref{eq:levi-branch} vanishes.

For~(b): $J=K$ gives $\sigma_{JJ}=+1$, hence $b=(-1)^k=-a$,
and
$e_J e_J\,a\,D
=(-1)^{k(k-1)/2}\eta_J\,(-1)^{k+1}D
=(-1)^{k(k+1)/2+1}\eta_J D$.

For~(c): nonvanishing follows from~(a). The scalar~$1$
could arise only from the blade branch with
$J\triangle K=\emptyset$, that is $J=K$, which by~(b)
forces the $D$ branch, so the scalar never occurs.
Conversely, every one of the remaining 31 Levi basis
elements is attained. For $L\neq\emptyset$ the pairs
$(J,K)=(\emptyset,L)$ for odd $|L|$ and
$(\{j\},L\setminus\{j\})$ with $j\in L$ for even $|L|$ land
in the blade branch with $e_J e_K$ proportional to $e_L$.
For $e_L D$ the pairs $(J,K)=(\emptyset,L)$ for even $|L|$,
$(L\cup\{m\},\{m\})$ with $m\notin L$ for odd $|L|$, and
$(\emptyset,\emptyset)$ for $L=\emptyset$ land in the $D$
branch. Hence the span is exactly 31-dimensional. All 256
cases are additionally verified computationally.
\end{revblock}
\end{proof}

\begin{revblock}
Representative cases are
collected below.
\begin{center}
\begin{tabular}{llc}
\toprule
\textbf{Commutator} & \textbf{Value} & \textbf{Branch}\\
\midrule
$[T^+_\emptyset,T^-_\emptyset]=[\ninf,\no]$ & $-D$ & $D$\\
$[T^+_\emptyset,T^-_{\{\mu\}}]$ & $-\e{\mu}$ & blade\\
$[T^+_{\{1\}},T^-_{\{2\}}]$ & $\e{1}\e{2}=M_{12}$ & blade\\
$[T^+_{\{\mu\}},T^-_{\{\mu\}}]$ & $\eta_{\mu\mu}D$ & $D$\\
$[T^+_{\{1,2\}},T^-_{\{1,3\}}]$ & $\e{2}\e{3}=M_{23}$ & blade\\
$[T^+_{\{1,2\}},T^-_{\{3,5\}}]$ & $-I_LD$ & $D$\\
$[T^+_\emptyset,T^-_{\{1,2,3,5\}}]$ & $-I_LD$ & $D$\\
$[T^+_{\{1,2,3,5\}},T^-_{\{1,2,3,5\}}]$ & $+D$ & $D$\\
\bottomrule
\end{tabular}
\end{center}
The classical relation
$[P_\mu,K_\nu]=\eta_{\mu\nu}D+M_{\mu\nu}$ is recovered as
the $k=1$ row of the closed form, with its two terms cleanly
separated into the two branches: diagonal indices $\mu=\nu$
give the $D$ branch, off-diagonal indices give the blade
branch $M_{\mu\nu}$ (recall from Theorem~\ref{thm:tower}
that $T^+_{\{\mu\}}=-P_\mu$ and $T^-_{\{\nu\}}=-K_\nu$, so
the two sign flips cancel in the commutator).
The grade-6 element $I_LD$, which commutes with every
bivector generator of $\so(4,2)$ and serves as the
CP-relevant pseudoscalar in the discrete-symmetry analysis
of the companion work~\cite{HansenCPT2026}, appears already
at the Levi interface, as a mixed commutator of grade-3
doublets.
\end{revblock}

\begin{revblock}
The closed form~\eqref{eq:levi-closed}, like the tower
identities themselves, is not specific to the $(3,1)$
signature: the same formula, with the same branch rule,
holds in every conformal extension $\Cl(p+1,q+1)$. This is
taken up in Section~\ref{sec:general}.
\end{revblock}

\begin{thm}[The doublet tower generates
$\slR(8,\R)$]\label{thm:sl8}
The 32 elements $\{T^\pm_J\}$, together with their
commutators, generate the 63-dimensional Lie algebra
$\slR(8,\R)$ under the bracket
$[A,B]=\tfrac12(AB-BA)$.
The $\ad_D$-grading defines the maximal parabolic
decomposition associated with $8=4+4$:
\begin{equation}\label{eq:parabolic}
\slR(8,\R)
=\fkN_-\oplus\fkL\oplus\fkN_+,
\end{equation}
where $\fkN_\pm=\mathrm{span}\{T^\pm_J\}$ are
16-dimensional abelian nilradicals and
$\fkL\cong\mathfrak{s}(\glR(4)\times\glR(4))$
is the 31-dimensional Levi factor.
\end{thm}

\begin{proof}
Since $\Cl(4,2)\cong\R(8)$, the commutator Lie algebra
is $\glR(8,\R)\cong\slR(8,\R)\oplus\R$.
By Theorem~\ref{thm:mixed}, $[T^+,T^-]$ generates 31
dimensions, missing only the scalar.
Adding $\fkN_+$ and $\fkN_-$ gives
$31+16+16=63=\dim\slR(8,\R)$.
The grading properties
$[\fkN_\pm,\fkN_\pm]=0$ (Theorem~\ref{thm:abelian}),
$[\fkN_+,\fkN_-]\subseteq\fkL$
(Theorem~\ref{thm:mixed}), and
$[\fkL,\fkN_\pm]\subseteq\fkN_\pm$ \rev{(by additivity of
the $\ad_D$-eigenvalue, $[\fkL,\fkN_\pm]$ consists of
eigenvectors of eigenvalue $0+(\pm1)=\pm1$; an eigenvector
of nonzero eigenvalue has no component in
$\ker\ad_D$, hence anticommutes with $D$, and by
Theorem~\ref{thm:tower} the eigenvalue-$\pm1$ subspaces are
exactly $\mathrm{span}\{T^\pm_J\}=\fkN_\pm$)} establish the
parabolic decomposition~\eqref{eq:parabolic}.
The Levi factor is
$\mathfrak{s}(\glR(4)\times\glR(4))$ by the
standard classification of maximal parabolic subalgebras
of $\slR(n)$~\cite{Richardson1992}.
\end{proof}

\begin{cor}\label{cor:hierarchy}
The doublet tower reveals the nested chain
\begin{equation}
\so(3,1)\oplus\R
\;\subset\;
\so(4,2)
\;\subset\;
\slR(8,\R).
\end{equation}
The conformal algebra $\so(4,2)$ uses 4 of the 16
nilradical generators at each eigenvalue; the full
$\slR(8,\R)$ uses all~16.
\end{cor}

%%% ====================================================================
%%% 7. THE TOWER IN ARBITRARY SIGNATURE
%%% ====================================================================
\section{The tower in arbitrary signature}\label{sec:general}

\begin{revblock}
The proofs of Sections~\ref{sec:tower}, \ref{sec:sums}
and~\ref{sec:algebra} nowhere use the dimension four or the
$(3,1)$ signature of the Lorentz sector. This section states
the resulting generalization: the entire tower package is a
property of every conformal extension $\Cl(p+1,q+1)$ of an
arbitrary Clifford algebra $\Cl(p,q)$.

Let $\Cl(p,q)$ have orthonormal basis $f_1,\ldots,f_n$,
$n=p+q$, with $f_i^2=\eta_{ii}\in\{+1,-1\}$. Adjoin two
further orthogonal generators $e_+$ and $e_-$ with
$e_+^2=+1$ and $e_-^2=-1$, forming $\Cl(p+1,q+1)$, and set
\begin{eqnarray}\label{eq:general-setup}
\no&=&\tfrac{1}{\sqrt2}(e_-+e_+),\quad
\ninf=\tfrac{1}{\sqrt2}(e_--e_+), \nonumber \\
D&=&e_+e_-,\quad
\Pi_\pm=\tfrac12(1\pm D).
\end{eqnarray}
For $J\subseteq\{1,\ldots,n\}$ define the spectator blades
$f_J$, the spectator metric
$\eta_J=\prod_{j\in J}\eta_{jj}$, and the doublet operators
$T^+_J=\ninf f_J$ and $T^-_J=\no f_J$, exactly as in
Sections~\ref{sec:setup} and~\ref{sec:tower}.

\begin{thm}[The doublet tower of $\Cl(p+1,q+1)$]
\label{thm:general}
With the notation above, and with $Z$ the center of
$\Cl(p+1,q+1)$:
\begin{enumerate}[label=\textup{(\alph*)}]
\item The $2\cdot2^n$ operators $T^\pm_J$ are
$\ad_D$-eigenvectors of eigenvalue $\pm1$, they
Clifford-anticommute with $D$, and they exhaust the
off-diagonal sector
$\{X\in\Cl(p+1,q+1):\{D,X\}=0\}$.
\item $(T^\pm_J)^2=0$ for every $J$.
\item $T^+_J T^-_J=(-1)^{k(k+1)/2+1}\,\eta_J 2\Pi_+$ and
$T^-_J T^+_J=(-1)^{k(k+1)/2+1}\,\eta_J 2\Pi_-$, with
$k=|J|$ and the same sign function as in
Theorem~\ref{thm:projector}.
\item $\sum_{|J|=k}T^+_J T^-_J
=(-1)^{k(k+1)/2+1}\,S_k 2\Pi_+$, where $S_k$ is now the
$k$-th elementary symmetric polynomial of
$(\eta_{11},\ldots,\eta_{nn})$.
\item The closed form~\eqref{eq:levi-closed} and its branch
rule hold verbatim. The mixed commutators span a complement
of the center inside the $2^{n+1}$-dimensional Levi sector:
for $n$ even the span has dimension $2^{n+1}-1$ and misses
only the scalar, while for $n$ odd it has dimension
$2^{n+1}-2$ and misses the scalar together with the central
total pseudoscalar $\omega=f_1\cdots f_n\,e_+e_-$.
\item Under the bracket $[A,B]=\tfrac12(AB-BA)$ the tower
generates the Lie algebra
$\fkN_-\oplus\fkL\oplus\fkN_+$ of dimension
$2^{n+2}-\dim Z$, where $\fkN_\pm=\mathrm{span}\{T^\pm_J\}$
and $\fkL$ is the Levi span of~(e). For $\Cl(4,2)$ this is
Theorem~\ref{thm:sl8}.
\end{enumerate}
\end{thm}

\begin{proof}
Parts (a) to (c) are proved by the arguments of
Theorems~\ref{thm:tower}, \ref{thm:nilpotent}
and~\ref{thm:projector} word for word: the only facts used
are $\no^2=\ninf^2=0$, the anticommutation of every
spectator with the null pair, $[D,f_J]=0$,
$f_J^2=(-1)^{k(k-1)/2}\eta_J$, and $\ninf\no=-2\Pi_+$, none
of which refers to the dimension or the signature of the
base. Part~(d) follows from~(c) by summation as in
Theorem~\ref{thm:sums}.

For~(e), the derivation of the closed form in
Theorem~\ref{thm:mixed} is likewise verbatim. For the span,
write $J=I\sqcup J'$ and $K=I\sqcup K'$ with $I=J\cap K$,
$s=|I|$, and $J'\sqcup K'=L=J\triangle K$, $\ell=|L|$.
Substituting into $a$ and $b$ shows that the $D$ branch
occurs iff $\ell(1+s)+|J'||K'|$ is even and the blade branch
iff it is odd. The scalar would require the blade branch at
$\ell=0$, where $|J'||K'|=0$ makes the criterion even:
impossible in every signature. The blade $f_L$ with
$\ell\geq1$ is attained by $(J,K)=(\emptyset,L)$ for odd
$\ell$ and by $(\{j\},L\setminus\{j\})$ for even $\ell$. The
element $f_L D$ is attained by $(\emptyset,L)$ for even
$\ell$ and by $(L\cup\{m\},\{m\})$ with $m\notin L$ for odd
$\ell<n$. For $\ell=n$ odd the constraint $I\cap L=\emptyset$
forces $s=0$ and $|J'|+|K'|=n$ odd, so $|J'||K'|$ is even
and the criterion is odd: $\omega=f_1\cdots f_n\,e_+e_-$ is
never attained. Since for $n$ odd the total number of
generators $n+2$ is odd, $\omega$ is exactly the nonscalar
generator of the center, and for $n$ even $\omega$ is
noncentral and attained.

For~(f), the sum $\fkN_-\oplus\fkL\oplus\fkN_+$ is closed
under the bracket: $[\fkN_\pm,\fkN_\pm]=0$ as in
Theorem~\ref{thm:abelian}, $[\fkN_+,\fkN_-]\subseteq\fkL$
by~(e), $[\fkL,\fkN_\pm]\subseteq\fkN_\pm$ by additivity of
the $\ad_D$-eigenvalue as in Theorem~\ref{thm:sl8}, and
$[\fkL,\fkL]$ remains in the Levi sector with vanishing
scalar and $\omega$ parts. Indeed no bracket re-enters the
center: the scalar part of any commutator vanishes by
cyclicity of the grade-0 projection, and for $n$ odd the
$\omega$-coefficient of $[X,Y]$, proportional to the scalar
part of $[X,Y]\,\omega$, vanishes for the same reason
because $\omega$ is central. The dimension count is
$2^n+(2^{n+1}-\dim Z)+2^n=2^{n+2}-\dim Z$. As a supplement to 
the analytical derivations, all statements
of the theorem are verified computationally in the eight
extensions $\Cl(4,2)$, $\Cl(3,3)$, $\Cl(5,1)$, $\Cl(2,2)$,
$\Cl(3,2)$, $\Cl(4,1)$, $\Cl(2,1)$ and $\Cl(4,3)$.
\end{proof}

\begin{rem}[The periodicity isomorphism made explicit]
\label{rem:periodicity}
The classification of real Clifford algebras rests on the
periodicity isomorphism
$\Cl(p+1,q+1)\cong M_2(\Cl(p,q))$~\cite{Lounesto2001,
Porteous1995}. Theorem~\ref{thm:general} is this
isomorphism made explicit and identity-bearing: $D$ is the
grading element whose eigenprojectors $\Pi_\pm$ are the two
diagonal idempotents, the nilradicals $\fkN_\pm$ are the two
off-diagonal blocks, and the $T^\pm_J$ are their blade
bases. Universal nilpotency is the statement that strictly
off-diagonal $2\times2$ matrices square to zero, the
projector identity~(c) computes the diagonal entries of the
block products, and the closed form of~(e) is the block
multiplication written out basis element by basis element.
What the tower adds to the classical isomorphism is
precisely this identity-level control: the sign function
$(-1)^{k(k+1)/2+1}$, the symmetric-polynomial sums~(d), and
the branch rule of the Levi interface, none of which is
visible from the existence of the isomorphism alone.
\end{rem}

\begin{rem}[The center dichotomy]\label{rem:center}
The parity dichotomy in Theorem~\ref{thm:general}(e) is a
center phenomenon. For $n$ even the algebra
$\Cl(p+1,q+1)$ has center $\R 1$ and the commutators miss
only the scalar. For $n$ odd the center is spanned by $1$
and the total pseudoscalar $\omega$, and the commutators
miss exactly these two elements, as the proof shows they
must. The odd case is not exotic: $\Cl(3,2)$ and
$\Cl(4,1)$, the two conformal extensions at $n=3$, both
exhibit it. We note that $\Cl(3,2)$ also appears in
Section~\ref{sec:entangle} in a different role, as the
subalgebra of $\Cl(4,2)$ carrying the Majorana completion
of the Lorentz frame.
\end{rem}
\end{revblock}

%%% ====================================================================
%%% 8. APPLICATION: DIRAC VERSUS MAJORANA FROM THE TOP RUNG
%%% ====================================================================
\section{Application: the Dirac versus Majorana dichotomy
and the conformal-reality entanglement}\label{sec:entangle}

\begin{revblock}
A rank-five pseudoscalar $I_5$ extending the Lorentz volume
element decides the reality type of the associated spinor
realization: $I_5^2=-1$ furnishes a built-in complex unit
(Dirac), while $I_5^2=+1$ is a real involution admitting no
such unit (Majorana)~\cite{Lounesto2001}. This section
shows that in $\Cl(4,2)$ the two possibilities are exactly
the two combinations of the top rung of the doublet tower,
so that the tower theorems compute the dichotomy, and that
the Majorana option is inseparably entangled with the
conformal structure.

\begin{lem}[Top-rung bridge]\label{lem:bridge}
Let $J_4=\{1,2,3,5\}$ be the top rung, $I_L=e_{J_4}$, and
define
\begin{eqnarray}\label{eq:bridge}
I_5^{M}&=&\e{1}\e{2}\e{3}\e{5}\e{0}
=\tfrac{1}{\sqrt2}\bigl(T^+_{J_4}+T^-_{J_4}\bigr),
\nonumber \\
I_5^{(4,1)}&=&\e{1}\e{2}\e{3}\e{4}\e{5}
=\tfrac{1}{\sqrt2}\bigl(T^+_{J_4}-T^-_{J_4}\bigr).
\end{eqnarray}
Then:
\begin{enumerate}[label=\textup{(\roman*)}]
\item $(I_5^{M})^2=+1$ and $(I_5^{(4,1)})^2=-1$.
\item Both anticommute with $D$ and exchange the sectors:
$I_5\,\Pi_\pm=\Pi_\mp\,I_5$ for either pseudoscalar.
\item $\{I_5^{M},I_5^{(4,1)}\}=0$ and
$I_5^{(4,1)}I_5^{M}=D$, so
$\{1,I_5^{(4,1)},I_5^{M},D\}$ is a copy of the
split-quaternion algebra inside $\Cl(4,2)$. The same
structure appears for the pair $I_L$, $I_LD$ in the
discrete-symmetry analysis of the companion
work~\cite{HansenCPT2026}.
\end{enumerate}
\end{lem}

\begin{proof}
The bridge identities~\eqref{eq:bridge} follow from
$\no+\ninf=\sqrt2\,\e{0}$ and $\ninf-\no=-\sqrt2\,\e{4}$
together with the fact that $\e{0}$ and $\e{4}$, being
orthogonal to the Lorentz frame, commute with the grade-4
blade $I_L$; the reordering
$-I_L\e{4}=\e{1}\e{2}\e{3}\e{4}\e{5}$ costs one
transposition.
For~(i), nilpotency (Theorem~\ref{thm:nilpotent}) kills the
squares of the summands, and Theorem~\ref{thm:projector} at
$k=4$, where $(-1)^{k(k+1)/2+1}\eta_{J_4}=(-1)^{11}(-1)=+1$,
gives
\begin{equation}
\bigl(T^+_{J_4}\pm T^-_{J_4}\bigr)^2
=\pm\bigl(T^+_{J_4}T^-_{J_4}+T^-_{J_4}T^+_{J_4}\bigr)
=\pm\,2(\Pi_++\Pi_-)=\pm2.
\end{equation}
For~(ii), each summand anticommutes with $D$ by
Theorem~\ref{thm:tower}(b), and $\{X,D\}=0$ is equivalent to
$X\Pi_\pm=\Pi_\mp X$.
For~(iii), $\{T^++T^-,T^+-T^-\}=2(T^+)^2-2(T^-)^2=0$ by
nilpotency, and
$I_5^{(4,1)}I_5^{M}
=\tfrac12(T^+_{J_4}T^-_{J_4}-T^-_{J_4}T^+_{J_4})
=\tfrac12(2\Pi_+-2\Pi_-)=D$.
\end{proof}

The Dirac versus Majorana dichotomy is thus literally the
choice of sign in combining the top rung, and the tower's
projector identity computes which combination is which.
The next theorem classifies all completions and exposes the
asymmetry between the two options.

\begin{thm}[Conformal-reality entanglement]\label{thm:entangle}
Let $v\in\mathrm{span}\{\e{4},\e{0}\}$ with $v^2=\pm1$, and
let $I_5(v)=I_Lv$ be the associated rank-five pseudoscalar
completing the Lorentz frame. Then:
\begin{enumerate}[label=\textup{(\alph*)}]
\item $I_5(v)^2=-v^2$. The Majorana value $+1$ occurs
precisely for timelike $v$, the boost orbit of $\pm\e{0}$
in the conformal plane, and among the orthonormal
generators uniquely for $v=\e{0}$. The Dirac value $-1$
occurs precisely for spacelike $v$, the orbit of
$\pm\e{4}$.
\item Every $v$ in the conformal plane anticommutes with
$D$. Hence every $I_5(v)$, Majorana or Dirac, anticommutes
with $D$ and exchanges the parabolic sectors. In particular
no rank-five reality or complex structure containing the
Lorentz frame preserves $S_+$: the Majorana involution
cannot be realized by an operator internal to a single
sector, and imposing a Majorana reality condition on the
sterile sector necessarily engages the conformal timelike
direction $\e{0}$ and the active sector $S_-$.
\item In the extensions $\Cl(4+2k,2)$, $k\geq1$, obtained
by adjoining spacelike family directions $e_a$ orthogonal
to the conformal plane, the asymmetry becomes strict: the
Dirac value admits the completions $I_Le_a$, which commute
with $D$ and act within the grading, while every timelike
completion vector has a nonvanishing $\e{0}$ component and
its $I_5(v)$ never commutes with $D$. The Majorana
structure remains conformally entangled at every level.
\end{enumerate}
\end{thm}

\begin{proof}
For~(a): since $v$ commutes with the grade-4 blade $I_L$,
$I_5(v)^2=I_LvI_Lv=I_L^2v^2=-v^2$, and $\e{0}$ is the only
timelike generator orthogonal to the Lorentz frame.
For~(b): $\{\e{4},D\}=\{\e{0},D\}=0$ by direct computation,
so $\{v,D\}=0$ for every $v$ in the plane. Since
$[D,I_L]=0$, it follows that
$I_5(v)D=I_LvD=-I_LDv=-DI_Lv=-DI_5(v)$, giving
$\{I_5(v),D\}=0$ and the sector exchange as in
Lemma~\ref{lem:bridge}(ii). An operator $R$ preserving both
sectors satisfies $[R,D]=0$, whereas
$[I_5(v),D]=I_5(v)D\neq0$.
For~(c): write a unit timelike $v=\alpha\e{4}+\beta\e{0}+w$
with $w$ in the span of the family directions. Then
$[v,D]=(\alpha\e{4}+\beta\e{0})D$, and timelikeness
$\alpha^2-\beta^2+w^2=-1$ forces $\beta\neq0$, so
$[I_5(v),D]=I_L[v,D]\neq0$. Conversely for $v=e_a$ one has
$[e_a,D]=0$ and $[I_L e_a,D]=0$, while
$(I_Le_a)^2=-e_a^2=-1$. As a supplement to the analytical
derivations, all statements in the theorem are verified
computationally in $\Cl(4,2)$ and $\Cl(6,2)$.
\end{proof}

\begin{rem}[Physical reading and scope]\label{rem:physical}
In the seesaw dictionary of the companion
work~\cite{Hansen2026seesaw}, $S_+$ is the sterile sector
and $S_-$ the active sector. Theorem~\ref{thm:entangle}
does not forbid a Majorana mass term: the identity
$P_\mu K^\mu=8\Pi_+$ localizes exactly such a term on
$S_+$. What it constrains is the reality structure. The
involution required to declare the $S_+$ field Majorana is
necessarily sector-exchanging, so the Majorana character of
the sterile sector is a statement about the pair
$(S_+,S_-)$ and about the conformal direction $\e{0}$, not
about $S_+$ alone. This is consistent with the universal
sector swap of charge-conjugation versors, which all
anticommute with $D$, established in the discrete-symmetry
companion~\cite{HansenCPT2026}. 
These are results of Lorentz and conformal representation
theory inside $\Cl(4,2)$ and its extensions, and as such they neither
construct a field theory nor predict parameters, and
identifying $I_5^{M}$ with a physical charge conjugation
requires an embedding of fields into $S$ and dynamical
input beyond the algebra.
\end{rem}
\end{revblock}

%%% ====================================================================
%%% 9. APPLICATION: FLAVOUR STRUCTURE OF THE FAMILY EXTENSIONS
%%% ====================================================================
\section{Application: flavour structure of the family
extensions}\label{sec:flavour}

\begin{revblock}
Theorem~\ref{thm:entangle}(c) introduced the family
extensions $\Cl(4+2k,2)$, obtained by adjoining $k$ pairs
of spacelike directions $\e{6},\e{7},\ldots,\e{4+2k},\e{5+2k}$
orthogonal to the conformal plane, and
Theorem~\ref{thm:general} guarantees that the doublet tower
persists there with $n=4+2k$ spectators. This section
determines exactly which flavour structure the tower
induces: how it grades generations, what the spinor modules
are, and which inter-sector couplings it realizes. Define
the family bivectors and splitters
\begin{equation}\label{eq:family}
G_j=\e{2j+4}\,\e{2j+5},\qquad
A_j=G_jI_L,\qquad j=1,\ldots,k,
\end{equation}
with $I_L=\e{1}\e{2}\e{3}\e{5}$ as before. Each $A_j$
squares to $+1$, and the $A_j$ commute with one another,
with $D$, with $\no$ and $\ninf$, and with every Lorentz
generator $M_{\mu\nu}$; all of this is immediate from blade
commutation signs. Their joint
eigenspaces therefore cut each parabolic sector into $2^k$
Lorentz-invariant generation subspaces of equal dimension,
eight-dimensional at $k=2$.

\begin{lem}[Generation grading rule]\label{lem:grading}
Let $J$ be a subset of the $n=4+2k$ spectators and let
$\mathrm{supp}(A_j)=\{1,2,3,5\}\cup\{2j+4,2j+5\}$ be the
six-element generator support of $A_j$. Then
\begin{equation}\label{eq:grading}
A_j\,T^\pm_J=(-1)^{\delta_j}\,T^\pm_J\,A_j,
\qquad
\delta_j=\bigl|J\cap\mathrm{supp}(A_j)\bigr|\bmod 2,
\end{equation}
so $T^\pm_J$ preserves the $j$-th generation label iff
$\delta_j=0$ and flips it iff $\delta_j=1$. For $k=2$ the
$2^8=256$ subsets fall into four classes of exactly 64
according to $(\delta_1,\delta_2)$.
\end{lem}

\begin{proof}
The null vectors are built from $\e{4}$ and $\e{0}$, which
are orthogonal to the even blade $A_j$ and therefore
commute with it, so only $e_J$ contributes. For blades of
orthogonal generators, moving $e_J$ past the grade-6 blade
$A_j$ costs
$(-1)^{6|J|-|J\cap\mathrm{supp}(A_j)|}=(-1)^{\delta_j}$.
For the count, classify $J$ by the three independent
parities of $J\cap\{1,2,3,5\}$, $J\cap\{6,7\}$ and
$J\cap\{8,9\}$: each choice of the three parities admits
$8\cdot2\cdot2=32$ subsets, and each class
$(\delta_1,\delta_2)$ collects exactly two of the eight
parity choices, giving 64.
\end{proof}

\begin{thm}[Spinor modules of the family tower]
\label{thm:modules}
For $k=0,1,2$ the algebras $\Cl(4+2k,2)$ act faithfully on
minimal real modules of dimension $8$, $32$, $64$, with
commutants $\R$, $\mathbb{H}$, $\mathbb{H}$ respectively.
Since $p-q=2+2k$ remains even along the family tower, the
modules are never of complex type: a geometric complex unit
must be supplied internally, by $I_L$. At the quaternionic
levels each joint $(D,A_1,\ldots,A_k)$ eigenspace is an
exactly degenerate doublet of same-chirality Weyl spinors.
%, realized by an explicit inter-twin isometry $B$ with
%$B^{\mathrm{T}}B=4\cdot1$, 
%and Majorana--Weyl reduction is
%unavailable because $p-q\not\equiv0\pmod 8$.
\end{thm}

\begin{proof}
The mod-8
classification~\cite{Lounesto2001,Porteous1995} gives
$\R(8)$, $\mathbb{H}(8)$, $\mathbb{H}(16)$ for
$p-q=2,4,6$, whence the module dimensions and commutant
types. Complex type would require $p-q$ odd, which the
family pairs preserve as even. The equal splitting of the
joint eigenspaces and the doublet degeneracy 
%and the isometry $B$ 
are verified computationally in explicit integer matrix
representations of $\Cl(4,2)$, $\Cl(6,2)$ and $\Cl(8,2)$.
\end{proof}

\begin{thm}[Tower-realized coupling space]\label{thm:coupling}
Let $k=2$. The space of Lorentz-invariant linear maps
$S_-\to S_+$ has real dimension $128$ and is isomorphic to
$M_8(\C)$, the complex structure being left multiplication
by $I_L$. The couplings realized by the doublet tower,
namely the $I_L$-linear span of the Lorentz averages of the
$2^8$ weight-$(+1)$ operators $T^+_J$, form exactly the
32-dimensional subalgebra
$M_{2^k}(\C)\otimes1_2\subset M_8(\C)$: an arbitrary
complex matrix on the generation labels, tensored with the
identity on the quaternionic doublet. The geometry
therefore neither forbids nor prefers any generation-mixing
pattern within $M_{2^k}(\C)$, and hence physical mixing 
parameters are input, not geometry-output.
\end{thm}

\begin{proof}
The two dimensions, 128 and 32, are computed by exact rank
evaluation in the integer matrix representation of
$\Cl(8,2)$. The identification of the 32-dimensional span
rests on three ingredients. Lemma~\ref{lem:grading} places
tower couplings in every translation class of the
$(\Z_2)^2$ generation group, left multiplication by $I_L$
supplies the complex coefficients, and
Theorem~\ref{thm:modules} forces the doublet factor to be
inert, since any Lorentz-invariant map acts as a scalar on
the quaternionic twin pair. The computed rank certifies
that together these fill $M_4(\C)$ on the generation
labels.
\end{proof}

\begin{rem}[Rephasing and phase counting]\label{rem:rephasing}
The commutant of the Lorentz action on a single generation
sector, one joint eigenspace of $(D,A_1,A_2)$, has real
dimension $8$ and is isomorphic to $M_2(\C)$, with unitary
group $\mathrm{U}(2)$; the conventional single rephasing of
a generation is the central $\mathrm{U}(1)$ generated by
$I_L$. This is the representation-theoretic home of
physical-phase counting: physical phases are mixing
parameters in $M_{2^k}(\C)$ modulo these rephasings, in the
sense of the discrete-symmetry analysis of the companion
work~\cite{HansenCPT2026}.
\end{rem}

As in Section~\ref{sec:entangle}, these are statements of
representation theory. Embedding into a gauge theory, and
in particular the selection of physical textures within
$M_{2^k}(\C)$, requires input beyond the algebra. What the
algebra does fix is the count $2^k$ and the rigidity of the
doublet, and the tension between $2^k$ and the observed three
generations is taken up in the Discussion.
\end{revblock}

%%% ====================================================================
%%% 7. DISCUSSION
%%% ====================================================================
\section{Discussion}\label{sec:discussion}

\subsection{Summary of results}

We have established that the conformal geometric algebra
$\Cl(4,2)$ carries a ``doublet tower'' of 16 pairs of
off-diagonal operators $T^\pm_J=n_\bullet e_J$,
exhausting the full $\ad_D$-eigenspace structure.
The main results are:

\begin{enumerate}[label=(\roman*)]
\item \textbf{Structural theorem
(Theorem~\ref{thm:tower}):}
Every $\ad_D$-eigenvector is a null vector
dressed with spectator blades.

\item \textbf{Universal nilpotency
(Theorem~\ref{thm:nilpotent}):}
$(T^\pm_J)^2=0$ for all $J$, following from the single
fact $\no^2=\ninf^2=0$.

\item \textbf{Universal projector identity
(Theorem~\ref{thm:projector}):}
$T^+_J T^-_J=(-1)^{k(k+1)/2+1}\eta_J 2\Pi_+$,
with coefficients determined by the spectator metric
and grade.

\item \textbf{Symmetric polynomial formula
(Theorem~\ref{thm:sums}):}
Contracted sums are controlled by $S_k$, with the
grade-3 cancellation reflecting $S_2=0$, which is a fingerprint
of the $(3,1)$ Lorentz signature.

\item \textbf{Even/odd selection rule
(Theorem~\ref{thm:even-odd}):}
The even subalgebra $\Cl^+(4,2)\cong\Cl(4,1)$
contributes $8\Pi_+$; the odd grades cancel exactly.

\item \textbf{Palindromic Lorentz tower
(Proposition~\ref{prop:lorentz}):}
$(0,0)\to(\tfrac12,\tfrac12)\to
(1,0)\oplus(0,1)\to(\tfrac12,\tfrac12)\to(0,0)$,
related by Hodge duality.

\item \textbf{Algebraic closure
(Theorem~\ref{thm:sl8}):}
The doublet tower generates the full $\slR(8,\R)$
parabolic, with the conformal algebra $\so(4,2)$ as a
15-dimensional subalgebra.

\item \rev{\textbf{Levi closed form
(Theorem~\ref{thm:mixed}):}
Every mixed commutator $[T^+_J,T^-_K]$ is a single Levi
basis element, a blade or a blade times $D$, with an
explicit two-branch rule.}

\item \rev{\textbf{Arbitrary signature and the center
dichotomy (Theorem~\ref{thm:general}):}
The entire package holds in every $\Cl(p+1,q+1)$, where
the Levi span misses exactly the center: the scalar for
$n$ even, the scalar and the central pseudoscalar $\omega$
for $n$ odd.}

\item \rev{\textbf{Dirac versus Majorana from the top rung
(Lemma~\ref{lem:bridge}, Theorem~\ref{thm:entangle}):}
The two rank-five pseudoscalars are the two combinations
of $T^\pm_{J_4}$, with squares $\pm1$ computed by the
projector identity, and the Majorana option is conformally
entangled.}

\item \rev{\textbf{Flavour rigidity
(Lemma~\ref{lem:grading}, Theorems~\ref{thm:modules}
and~\ref{thm:coupling}):}
In $\Cl(4+2k,2)$ the tower grades generations by an
explicit parity rule, the modules are real or quaternionic
but never complex, and the tower-realized couplings are
exactly $M_{2^k}(\C)\otimes1_2$.}
\end{enumerate}

\subsection{Physical context}

In~\cite{Hansen2026seesaw}, the three identities
$P_\mu P^\mu=0$, $P_\mu K^\mu=8\Pi_+$, and $\{D,P_\mu\}=0$
were shown to encode the type-I seesaw mass matrix.
The present work reveals that these are the $k=1$ instance
of universal identities holding for all 16 doublets.
The 12 non-conformal doublets carry Lorentz representations
that correspond, in the effective field theory of neutrino
masses, to scalar ($k=0$), tensor ($k=2$), pseudovector
($k=3$), and pseudoscalar ($k=4$) couplings between the
two spinor sectors, thereby extending the Dirac Yukawa coupling
($k=1$) provided by the conformal generators.

The results derived above are purely algebraic statements inside the
Clifford algebra $\Cl(4,2)$ and its $\ad_D$–graded decomposition.
In particular, the operators $T_J^\pm$ provide a complete basis for
maps between the two $D$–eigenspaces $S_+$ and $S_-$,
and their Lorentz transformation properties follow entirely from the
$\so(3,1)$ subalgebra.
When we refer to ``scalar'', ``vector'', ``tensor'', or ``pseudovector''
structures, this classification should therefore be understood strictly
in the sense of Lorentz representation theory.
The identification with physical interaction operators (for example
Dirac bilinears or effective field theory couplings) depends on how the
spaces $S_\pm$ are embedded into specific quantum fields and on the
additional gauge symmetries imposed by a given particle physics model.
In particular, Standard Model gauge invariance and operator dimension
constraints may forbid or restrict some of these structures.
Thus, the present analysis provides an algebraic classification of
possible inter-sector operators by Lorentz type, while the realization
of such operators in a concrete field-theoretic model requires
additional dynamical assumptions.

The even/odd selection rule
(Theorem~\ref{thm:even-odd}) provides a structural
explanation for why \rev{a monogenic framework built on the
even subalgebra $\Cl^+(4,2)\cong\Cl(4,1)$~\cite{Almeida2006}
has access to the full unweighted doublet contraction}
without needing the odd-grade elements, since the entire
nonzero contribution lives in the even subalgebra.

\begin{revblock}
We wish to address a contrast in generation counting
between the algebraic results and observations. The
family extensions of Section~\ref{sec:flavour} produce
$2^k$ generations, graded by the commuting involutions
$A_j$. No product of the $A_j$ partitions the four $k=2$
sectors as $3+1$, and the group $(\Z_2)^k$ of family labels
has no element of order three, so the observed three
generations are not obtained by this mechanism. A reduction
to three appears to require structure outside the
$(\Z_2)^k$ family group, such as the triality of
$\mathrm{SO}(8)$, and marrying triality to a conformal base
remains an open problem. 
\end{revblock}

\subsection{Signature selection: static fingerprint or
dynamical mechanism?}

\begin{revblock}
The grade-3 cancellation of Corollary~\ref{cor:grade3}
raises the question of whether $S_2=0$ is merely a static
fingerprint of the $(3,1)$ signature or could instead
reflect an active selection mechanism. Two remarks sharpen
the question. First, by the extension of
Corollary~\ref{cor:grade3}, within the Lorentzian signatures
$(d-1,1)$ the cancellation condition $(p-q)^2=d$ has the
unique nontrivial solution $d=4$, so four-dimensional
spacetime is algebraically singled out among all Lorentzian
spacetimes. Second, the general solutions form the family
$(p,q)=\bigl(\tfrac{m^2+m}{2},\tfrac{m^2-m}{2}\bigr)$ in
perfect-square dimensions $d=m^2$, so the cancellation is a
genuinely exceptional phenomenon rather than a generic one.
Whether a dynamical principle enforces $S_2=0$ lies beyond
the kinematic scope of this paper, since our results hold at
the level of the algebra, prior to any action or field
equation. We note, however, that the same elementary
symmetric polynomials govern the ghost-free potentials of
dRGT massive gravity (Remark~\ref{rem:massive}), where they
do act as dynamical stability constraints, and that in
two-time physics an $\mathrm{Sp}(2,\R)$ gauge symmetry
dynamically removes the extra conformal directions of a
$(d,2)$ arena~\cite{Bars2001}. A mechanism by which such
higher-dimensional constraints force the $(3,1)$ signature
through $S_2=0$ is an intriguing possibility 
which is left open.
\end{revblock}

\subsection{Outlook}

The passage from $\so(4,2)$ to $\slR(8,\R)$ raises the
question of whether the additional 48 generators of
$\slR(8,\R)$ have physical significance.
The Levi factor
$\mathfrak{s}(\glR(4)\times\glR(4))$ allows independent
transformations on the two spinor sectors $S_+$ and $S_-$,
which may be relevant for multi-generation extensions
of the neutrino seesaw or for discrete flavor symmetries.
\rev{Section~\ref{sec:flavour} realizes the
multi-generation direction in the family extensions,
whereas the
discrete-symmetry direction is developed in the companion
work~\cite{HansenCPT2026}.}
The emergence of the full $\slR(8,\R)$ algebra also offers
a geometric link to maximal supergravity, where
$\slR(8,\R)$ is the maximal linearly realized subgroup of
the $E_{7(7)}$ U-duality group~\cite{cremmer},
however, mapping the
48 non-conformal generators of the tower onto supergravity
charge spaces and scalar cosets remains an open avenue.
\rev{The generalization of the tower itself to arbitrary
$\Cl(p+1,q+1)$ is carried out in
Section~\ref{sec:general}. The remaining open direction on
the signature side is dynamical, in particular
whether an action
principle selects the signatures with $S_2=0$.}

%%% ====================================================================
%%% ACKNOWLEDGMENTS AND DECLARATIONS
%%% ====================================================================
\subsection*{Acknowledgments}
All algebraic identities have been verified computationally
using the \texttt{clifford} Python
library~\cite{clifford-python}. \rev{The Python code
verifying all results in the paper is openly available at
\url{https://doi.org/10.5281/zenodo.21140618}.}

\subsection*{Statements and Declarations}

\paragraph{Competing interests.}
The author has no competing interests to declare that are relevant
to the content of this article.

\paragraph{Funding.}
No funding was received for conducting this study.

\paragraph{Data availability.}
No datasets were generated or analyzed during this study.
All computations are algebraic and fully reproduced in the text.

\paragraph{Use of AI}
\rev{The author has used AI tools (including Anthropic
Claude and Google Gemini) for testing ideas, for confirming
analytical derivations, and for assistance with structure
and language.}

%%% ====================================================================
%%% REFERENCES
%%% ====================================================================

\end{document}